\documentclass[11pt]{amsart}
\usepackage{amsfonts}
\usepackage{amsthm}
\usepackage{amsmath}
\usepackage{amssymb}
\usepackage{enumerate}
\numberwithin{equation}{section}

\newtheorem{theorem}{Theorem}[section]

\newcommand{\mymod}[2]{#1~\text{mod}~#2}

\title[Generalised Function Field Ramanujan sums]{Generalisations of Ramanujan sums for Polynomial rings over finite fields}

\author{J. C. Andrade}
\address{Department of Mathematics, University of Exeter, Exeter, EX4 4QF, United Kingdom}
\email{j.c.andrade@exeter.ac.uk}

\author{J. R. P. Hanslope}
\email{jrph202@exeter.ac.uk}


\begin{document}

\begin{abstract}
In this paper, we consider a general form of the analogue of Ramanujan's sum in the ring of polynomials over a finite field. We first prove some multiplicative properties of such functions before considering their finite Fourier series and some specific examples. In the end we also prove a result about the Dirichlet series of such functions. 
\end{abstract}

\maketitle

\section{Introduction}
Anderson and Apostol in \cite{anderson_apostol_1} studied some multiplicative properties of functions $F$ and $S$ expressible as $F(k)=\sum_{d \mid k }f(d)h(k/d)$ or the more general $S(n;k)=\sum_{d \mid (n,k)}f(d)h(k/d)$ where $f$ and $h$ are arithmetical functions. In their paper they prove several results involving multiplicative properties of $S$ and $F$ and the relationship between these two functions. In the same paper they also consider some particular examples of such functions, including the Ramanujan sum

$$c_k(n)= \sum_{\substack{\mymod{m}{k}\\(m,k)=1}}\exp(2 \pi i nm/k) = \sum_{d \mid (n,k)} d \mu(k/d),$$ 
and then they compute the Fourier series for these functions. Finally, they establish formulas for the Dirichlet series involving the function $S$ and the Ramanujan sum. 

The main aim of this paper is to establish function field analogues of the results found by Anderson and Apostol. It is important to note that we need to extend and develop in some extension the Ramanujan sums in the context of function fields which is not a trivial matter. We also present several examples for the Fourier series expansion of some arithmetic functions in the context of polynomials over finite fields (for more details see Section 3.2).

\section{Multiplicative Properties}

Before we state and prove the main results of this paper, we need to introduce some notation. Let $ A = \mathbb{F}_q[X] $ be the ring of polynomials over the finite field with $q$ elements, where $q = p^t$ for some prime number $ p $ and positive integer $ t $. Let an arithmetical function be a complex valued function defined on the set $ \{ f \in A : f \text{ is monic}\}$ and let such a function be called \textit{multiplicative} if, for coprime monic polynomials $f$ and $g$, $ G(fg)=G(f)G(g).$ We say $G$ is \textit{completely multiplicative} if the previous identity holds for all monic polynomials $f,g$. We also define $G(1)=1$ for any multiplicative or completely multiplicative function $G$. 

\bigskip

As in the Anderson and Apostol paper \cite{anderson_apostol_1}, we will be considering functions of the form 

\begin{equation} \label{defn:S}
S(h;f)=\sum\limits_{g \mid (h,f)} G(g) H (f/g)
\end{equation} 
where $h$ and $f$ are monic polynomials in $A$ and $G$ and $H$ are arithmetical functions. Our first result is the following.

\begin{theorem} \label{thm:S_multi}
Suppose $S$ is defined as in (\ref{defn:S}), $G,H$ are multiplicative functions and $h_1, h_2, f_1, f_2 $ monic polynomials in $ A $. If 

\begin{equation}\label{assumptions:thm1} (h_1, h_2) = (f_1, f_2 ) = (h_1, f_2) = (h_2, f_1) = 1, 
\end{equation} then 

$$S(h_1 h_2;f_1 f_2)=S(h_1; f_1)S(h_2;f_2). $$
\end{theorem}
This can be seen as the function field analogue of Theorem 1 of \cite{anderson_apostol_1}. Before we can prove this theorem we need the following result. If the conditions  (\ref{assumptions:thm1}) hold, then 

$$ (h_1 h_2, f_1 f_2) = (h_1, f_1) (h_2, f_2). $$ 
We can easily verify this by first showing that the gcd function is multiplicative in one variable. 

\begin{proof}[Proof of Theorem 2.1]
Suppose that the assumptions in (\ref{assumptions:thm1}) hold. Let $a_1 = (h_1,f_1)$ and $a_2 = (h_2,f_2)$, then 

\begin{align*}
S(h_1 h_2, f_1 f_2) &= \sum \limits_{g \mid (h_1 h_2, f_1 f_2)} G(g)H\left(\frac{f_1 f_2}{g}\right)=\sum \limits_{g\mid a_1 a_2}G(g)H\left(\frac{f_1 f_2}{g}\right)
\end{align*}
Note that the conditions in (\ref{assumptions:thm1}) imply that $(a_1, a_2) = 1$ so any divisor of $a_1 a_2$ has the form $g_1 g_2$ where $g_1 \mid a_1 $, $g_2 \mid a_2 $ and $(g_1, g_2)=1$. Therefore (using the fact that $G$ and $H$ are multiplicative and $(f_1,f_2)=1$) 

\begin{align*}
S(h_1 h_2, f_1 f_2) &= \sum \limits_{g_1 g_2 \mid a_1 a_2} G(g_1 g_2) H\left(\frac{f_1 f_2}{g_1 g_2}\right) \\
&=\sum \limits_{g_1 \mid (h_1, f_1)} G(g_1)H\left(\frac{f_1}{g_1}\right)\sum \limits_{g_2 \mid (h_2, f_2)} G(g_2)H\left(\frac{f_2}{g_2}\right) \\
&= S(h_1;f_1)S(h_2;f_2).
\end{align*}
\end{proof}

Now let us also consider functions of the form

\begin{equation*}
F(f) = \sum \limits_{g \mid f} G(g)H\left(\frac{f}{g}\right)
\end{equation*} 
which is simply a particular case of (\ref{defn:S}) in which $f \mid h$. For the next theorem, we will need the M\"obius function $\mu$, defined as $\mu(f)=0$ if $f$ is not square free and $\mu(f)=(-1)^t$ if $f$ is a constant times a product of $t$ distinct monic irreducibles; note that $\mu(1)=1$. We now state and prove our next main result.

\begin{theorem}
\label{thm:2.2.2}
Let $H(h)=J(h)\mu(h)$ where $J(h)$ is multiplicative. Suppose also that $G(h)$ is completely multiplicative and that for any monic irreducible $P$, $G(P) \neq 0$ and $G(P) \neq J(P)$. Then, if $S(h;f)$ is defined as in (\ref{defn:S}), we can express $S$ in terms of $F$ as

\begin{equation} \label{thm:2}
S(h;f)=\left.J\left(\frac{f}{(h,f)}\right)\mu\left(\frac{f}{(h,f)}\right)F(f)\middle/F\left(\frac{f}{(h,f)}\right).\right.
\end{equation}
\end{theorem} 

\begin{proof}
First, let the expression on the right hand side of (\ref{thm:2}) be $R(h;f)$. Then, by the multiplicativity of $J,~\mu$ and $F$ we have that if the assumptions in (\ref{assumptions:thm1}) hold then 

$$R(h_1 h_2; f_1 f_2)=R(h_1;f_1) R(h_2 ; f_2).$$ 
Because of this and the result of Theorem \ref{thm:S_multi} we need only show that (\ref{thm:2}) holds in the case that $f=P^\alpha$ and $h=P^\beta$ for some monic irreducible $P$. Under these assumptions and that $\delta = \min\{\alpha, \beta\}$ (so $(k,h) = P^\delta$), we can minipulate the left hand side of (\ref{thm:2}) as

\begin{align}
S(P^\beta;P^\alpha)&=\sum\limits_{g\mid P^\delta}G(g)H(f/g) = \sum \limits_{i=0}^\delta G(P^i) J(P^{\alpha - i}) \mu (P^{\alpha - i}) \label{thm:2.1}
\end{align}
Whereas the right hand side of (\ref{thm:2}) becomes 

\begin{equation}
\left.J(P^{\alpha-\delta})\mu(P^{\alpha-\delta})F(P^\alpha)\middle/F(P^{\alpha - \delta} )\right. \label{thm:2.2}
\end{equation}
We will consider three cases: (i) $\alpha - \delta \geq 2 $, (ii)$\alpha - \delta = 0 $ and (iii) $\alpha - \delta = 1$. 
In (i) we have that $\mu(P^{\alpha-\delta}) = 0 $ so (\ref{thm:2.2}) vanishes. Also note that $\mu(P^{\alpha-i})=0 $ for all $0 \leq i \leq \delta $ so each term in (\ref{thm:2.1}) also vanishes. Now, in (ii), (\ref{thm:2.2}) becomes

$$ \left. J(1)\mu(1) F(P^\alpha) \middle/ F(1)=F(P^\alpha)\right.$$ 
On the other hand, (\ref{thm:2.1}) becomes 

\begin{align*}
\sum\limits_{i=0}^\alpha G(P^i) J(P^{\alpha-i})\mu(P^{\alpha-i})=\sum\limits_{g \mid P^\alpha} G(g) H(P^\alpha/g) = F(P^\alpha).
\end{align*}
So the result holds in case (ii). Finally, in case (iii), (\ref{thm:2.2}) becomes 

$$ J(P) \mu(P) F(P^\alpha)/F(P).$$
And the equation (\ref{thm:2.1}) becomes 
$$ \sum \limits_{i=0}^{\alpha-1} G(P^i)J(P^{\alpha-i}) \mu(P^{\alpha-i}). $$ 
We have that $\mu(P^{\alpha - i})$ will be zero unless $\alpha -i \leq 1 $, i.e.\ $i \geq \alpha - 1$. Thus it is only the very last term, $G(P^{\alpha-1})J(P)\mu(P)$, that is not zero. If $\alpha=1$ then we have $J(P)\mu(P)$ and $G(1)J(P)\mu(P)$ which are equal. Otherwise, if $\alpha>1$, then we have

\begin{align*}
F(P^\alpha)&=\sum\limits_{g \mid P^\alpha} G(g)H(P^\alpha/g) = \sum \limits_{i=0}^\alpha G(P^i)J(P^{\alpha-1})\mu(P^{\alpha-i}) \\
&= G(P^\alpha)+G(P^{\alpha-1})J(P) \\
&= G(P^{\alpha-1})(G(P)+J(P)),
\end{align*} 
since $G$ is completely multiplicative. Hence $F(P^\alpha)/F(P)=G(P^{\alpha-1})$. So then both (\ref{thm:2.1}) and (\ref{thm:2.2}) agree in case (iii) and therefore all three cases. 
\end{proof}

We now state and prove a third theorem, which is a multiplicative property of $F$.

\begin{theorem}
Suppose $H(h)$ and $G(h)$ satisfy the conditions of Theorem \ref{thm:2.2.2}. Then the function 

$$F(h) = \sum \limits_{g \mid h} G(g)J(h/g) \mu (h/g)$$ satisfies 

$$F(h h')=F(h)F(h')\frac{G((h,h'))}{F((h,h'))} $$ for any $h' \in A$.
\end{theorem}

\begin{proof}
Since $G$ is completely multiplicative

\begin{align*}
F(h) = \sum \limits_{g \mid h} G(h/g) J(g) \mu(g) = G(h) \sum \limits_{g \mid h} \mu(g) \frac{J(g)}{G (g)}.
\end{align*}
Note that we can write this sum as an Euler product where the product runs over all monic irreducible divisors of $h$ 

$$ \sum \limits_{g \mid h} \mu(g) \frac{J(g)}{G (g)} = \prod \limits_{P \mid h} \left(1-\frac{J(P)}{G(P)}\right).$$ 
Subbing this into the expression above we get 

$$ F(h) = G(h) \prod \limits_{P \mid h} \left(1-\frac{J(P)}{G(P)}\right),$$ and then

\begin{align*} \frac{F(hh')}{F(h)F(h')} &=\frac{G(hh')}{G(h)G(h')}\cdot \frac{\prod \limits_{P \mid hh'} \left(1-\frac{J(P)}{G(P)}\right)}{\prod \limits_{P \mid h} \left(1-\frac{J(P)}{G(P)}\right)\prod \limits_{P \mid h'} \left(1-\frac{J(P)}{G(P)}\right)} \\
&=\frac{1}{\prod \limits_{P \mid (h,h')} \left(1-\frac{J(P)}{G(P)}\right)} \\
&= \frac{G((h,h'))}{F((h,h'))}.
\end{align*} 
Rearranging gives the required result that

$$F(hh') = F(h)F(h')\frac{G((h,h'))}{G((h,h'))}.$$
\end{proof}

\section{Sums and finite Fourier Series}
In this section we first see that a particular example of functions of the type (\ref{defn:S}) is the function field analogue of the Ramanujan sum. We then look at expressing functions of the form (\ref{defn:S}) as Fourier series and consider some particular examples. 

\subsection{Stating and proving the theorems}
We first need to establish an analogue of the exponential function in function fields. Carlitz did some work on the polynomial Ramanujan sum in \cite{carlitz_1} so we follow \S 2 of that paper. Recall first that $q=p^t$ for some prime number $p$. Let $f,h \in \mathbb{F}_q[X]$ be monic polynomials. Suppose first that $t=1$ so then $q=p$. Write 

$$ h = \alpha_1 X^{r-1} + \ldots + \alpha_r \mod f$$ where $r = \deg(f)$. Then put 

$$\varepsilon(h,f)=\exp(2 \pi i \alpha_1 /p).$$ Now suppose that $t>1$. Let $GF(p^{t})$ be the Galois field of order $p^{t}$, defined by 

$$ \Psi(\theta)=\theta^t+c_1\theta^{t-1}+\ldots+c_t=0$$ were $\Psi(X)$ is an irreducible polynomial of $\mathbb{F}_p[X]$. Then let 

$$\alpha_1=\alpha_1(\theta) = a_1 \theta^{t-1}+\ldots+a_t,$$
where $\alpha_{1}$ is the leading coefficient of $h\pmod f$.
 
In this case, define 

$$\varepsilon(h,f) = \exp(2 \pi i a_1/p).$$ 
Now we are in a position to state and prove the analogue of equation (3) from \cite{anderson_apostol_1}.

\begin{theorem} \label{theorem:raman_analogue}
Let $\varepsilon(h,f)$ be defined as above. Then we define $\eta$ by 

$$ \eta (h,f) = \sum_{\substack{g~\text{mod}~f \\ (g,f)=1}} \varepsilon(hg,f). $$ 
Then we have that 

$$ \eta(h,f) = \sum_{d \mid (h,f)} |d| \mu\left(\frac{f}{d}\right).$$
\end{theorem}

Note that for $f \in A = \mathbb{F}_q[X]$ we have that $|f| = q^{\deg(f)}$.
Clearly the function $\eta(h,f)$ is the polynomial analogue of the Ramanujan sum and the result, the correct analogue of \cite[equation (3)]{anderson_apostol_1}.

\begin{proof}
Suppose we have some function $\delta$ such that $\delta(kg,kf)=\delta(g,f)$, where we are also using the assumption that $\delta(g,f)$ depends only on $g\pmod f$. Then every $\delta(g,f)$ such that $\deg(g) < \deg(f)$ may be expresed as $\delta(g',d)$ where $\deg(g')<\deg(d)$, $(g',d)=1$ and $d\mid f$. Then we have 

$$ \sum_{\mymod{g}{f}} \delta(g,f) = \sum_{d \mid f} \sum_{\substack{\mymod{g}{d} \\(g,d)=1}} \delta(g,d).$$ 
Letting $ \delta_1(f)=\sum\limits_{\mymod{g}{f}} \delta(g,f)$ 
and 
$\delta_2(f)=\sum\limits_{\substack{\mymod{g}{f}\\(f,g)=1}} \delta(g,f) $ gives 

$$\delta_1(f) = \sum_{d\mid f} \delta_2(d).$$ 
Using the M\"obius inversion formula gives 

$$\delta_2(f)=\sum_{d \mid f}  \mu\left(\frac{f}{d}\right) \delta_1(d).$$ 
So then we have 

$$\sum_{\substack{\mymod{g}{f}\\(f,g)=1}} \delta(g,f)=\sum_{d\mid f} \mu \left( \frac{f}{d} \right) \sum_{\mymod{g}{d}}\delta(g,d).$$
Now let $\delta(g,f)=\varepsilon(hg,f)$. 
From \cite[(2.5)]{carlitz_1} we see that this satisfies the required property of $\delta$. 
So then we have 

$$\sum_{\substack{\mymod{g}{f}\\(f,g)=1}} \varepsilon(hg,f)=\sum_{d\mid f} \mu \left( \frac{f}{d} \right) \sum_{\mymod{g}{d}}\varepsilon(hg,d).$$ 
From \cite[equation (2.6)]{carlitz_1} we have that 

\begin{equation} \label{equation:divisor_or_0} 
\sum_{\mymod{g}{f}}\varepsilon(hg,f)=
\begin{cases}
|f| & \text{if } f\mid h\\
0 & \text{if } f \nmid h.
\end{cases} 	
\end{equation} So we have
 
\begin{align*}
\eta(h,f) &= \sum_{d\mid f, d \mid h }  \mu \left(\frac{f}{d}\right)|d| \\
&= \sum_{d \mid(f,h)} |d| \mu \left(\frac{f}{d} \right).
\end{align*}
\end{proof}

We now look at the Fourier series of general functions of the form (\ref{defn:S}).

\begin{theorem}
For given arithmetic functions $G(h)$ and $H(h)$ we consider the function $S(h;f)$ as defined in (\ref{defn:S}). This function may be expressed as 

\begin{equation} \label{equation:S_sum_formula}
S(h;f)=\sum_{\mymod{g}{f}}a_f(g) \varepsilon(hg,f),
\end{equation} 
where the sum extends over any complete residue system modulo $f$, $\varepsilon$ is as defined above and the coefficients $a_f(g)$ are given by 

\begin{equation} \label{euqation:a_coeff}
a_f(g)=\frac{1}{|f|}\sum_{d \mid (f,g) } |d|H(d)G\left(\frac{f}{d}\right).
\end{equation} 
Also, if $G(h)$ is completely multiplicative and $r=(g,f)$, we have 

$$ a_f(g)=\frac{1}{|f|}G\left(\frac{f}{r}\right)\sum_{d \mid r} |d| H(d) G\left(\frac{r}{d}\right) $$
\end{theorem}

\begin{proof}
We work by first defining $a_f(g)$ as in (\ref{euqation:a_coeff}) and then deriving (\ref{equation:S_sum_formula}). Consider the sum 

\begin{align*}\sum_{\mymod{g}{h}}a_f(g)\varepsilon(hg,f) &=\sum_{\mymod{g}{f}}\frac{1}{|f|}\sum_{d\mid(f,g)}|d|H(d)G\left(\frac{f}{d}\right)\varepsilon (hg,f)\\
&= \frac{1}{|f|} \sum_{\mymod{g}{f}}\sum_{d\mid f,d \mid g} |d|H(d)G(f/d) \varepsilon(hg,f).
\end{align*}
Writing $g=g'd$ we get 

$$\frac{1}{|f|} \sum_{\mymod{g'd}{f}} \sum_{d \mid f} |d|H(d)G\left(\frac{f}{d}\right)\varepsilon(hg'd,f)$$ 
and then as $d\mid f$ we can rewrite this as 

$$\sum_{d \mid f} H(d) \frac{G(f/d)}{|f/d|} \sum_{\mymod{g'}{f/d}}\varepsilon(hg',f/d).$$ 
Now, from (\ref{equation:divisor_or_0}), we have that the sum on the right is either $|f/d|$ or $0$ depending on whether $(f/d)\mid h$ or not. Thus we have 

$$ \sum_{d \mid f, (f/d) \mid h} H(d) G(f/d)=\sum_{d\mid (f,h)}G(d)H(f/d)=S(h;f).$$ 
Note that, as $(g+f, f)=(g,f)$, we have that $a_f(g+f)=a_f(g)$ and so the particular residue system modulo $f$ is not important. Finally, suppose that $G(h)$ is completely multiplicative, then letting $r=(g,f)$, we have 

$$G\left(\frac{f}{d}\right)=G\left(\frac{f}{r}\frac{r}{d}\right)=G\left(\frac{f}{r}\right)G\left(\frac{r}{d}\right)$$ and the final result holds.
\end{proof}

\subsection{Some examples}
In this section we take $G(f)=|f|$ and evaluate the coefficients $a_f(g) $ for various functions $H(f)$. Note that the coefficients will be given by 

$$ a_f(g) = \sum_{d\mid(f,g)}H(d).$$
We first do $H(f) =\mu(f)$. Then we get
$a_f(g)=\sum_{d \mid (f,g)} \mu(d). $
Note that if $(f,g)=1$, the only divisor is $1$, $\mu(1)=1$ so in this case $a_f(g)=1$. If $(f,g)>1$, suppose $(f,g)=P_1^{\alpha_1}\ldots P_r^{\alpha_r}$ for monic irreducibles $P_i$. Then 

\begin{align*}
a_f(g) &= \mu(1)+\mu(P_1)+\ldots+\mu(P_r)+\mu(P_1 P_2)+\ldots+\mu(P_{r-1}P_r)+\ldots+\mu(P_1\ldots P_r) \\
&= \binom{r}{0}+\binom{r}{1}(-1)+\binom{r}{2}(-1)^2+\ldots+\binom{r}{r}(-1)^r =0
\end{align*} Thus 

\begin{align*}
a_f(g)=
\begin{cases}
1 & \text{if $(f,g)=1$}\\
0 & \text{if $(f,g)>1$}
\end{cases}
\end{align*}
and so 

\begin{equation}
\sum_{\substack{\mymod{g}{f}\\(f,g)=1}} \varepsilon(hg,f)=\sum_{g \mid (h,f)}|g|\mu(f/g).
\end{equation} 
Note that this is just confirming the result found by Theorem \ref{theorem:raman_analogue}.\bigskip

Now take $H(h)=\Phi(h)$, where $\Phi(f)$ is the polynomial equivalent of the Euler totient function defined as the number of non-zero polynomials of degree less than $\deg(f)$ and coprime to $f$. Note the identity, for $h \in  A=\mathbb{F}_q[X]$ 

\begin{equation}\label{equ:phi_ident}
\sum_{d\mid h} \Phi(d)=|h|.
\end{equation} 
Thus $a_f(g)=|(f,g)|$ and so 

\begin{equation}
\sum_{g \mod f} |(f,g)| \varepsilon(hg,f)=\sum_{g\mid(h,f)}|g|\Phi(f/g).
\end{equation}

Now take $H(h)=\mu(h)/|h|$. From (\ref{equ:phi_ident}) we see that $\sum_{d \mid h} \Phi \left({h}/{d}\right)=|h|$ and then by the M\"obius inversion formula $\sum_{d \mid h} \mu(d){|h|}/{|d|}=\Phi(h)$. Finally, rearranging and letting $h=(f,g)$, we have 

$$a_f(g)=\frac{\Phi((f,g))}{|(f,g)|}.$$ 
Rearranging the equation for $S(h;f)$ we get 

\begin{equation}
|f|\sum_{\mymod{g}{f}} \frac{\Phi((f,g))}{|(f,g)|}\varepsilon(hg,f)=\sum_{g \mid (f,h)}|g|^2 \mu(f/g).
\end{equation}

Now let $H(f)=\frac{\mu^2(f)}{\Phi(f)}$. 
So $a_f(g)=\sum_{d \mid(f,g)} \frac{\mu^2(d)}{\Phi(d)}.$
Writing this sum as an Euler product and rearranging gives the following.

\begin{align*} a_f(g) &= 
 \prod_{P\mid (f,g)}\left(1+\frac{\mu^2(P)}{\Phi(p)}\right) = \prod_{P \mid (f,g)} \left(1-\frac{1}{|P|}\right)^{-1}   \\
&= \frac{|(f,g)|}{\Phi((f,g))} \qquad\text{by \cite[Proposition 2.4]{rosen_1}}
\end{align*}
Thus 

\begin{equation} \sum_{\mymod{g}{f}}\frac{|(f,g)|}{\Phi((f,g))}\varepsilon(hg,f)=\sum_{g \mid(h,f)}|g|\left.\mu^2\left(\frac{f}{g}\right)\middle/ \Phi\left(\frac{f}{g}\right)\right.
\end{equation}

Now we will let $H(f)$ be the polynomial Von Mangoldt function, $\Lambda$, which we define as 

\begin{align*}
\Lambda(f)=\begin{cases} \log_q|P| & \text{if $f=P^k$ for monic irreducible $P$ and positive integer $k$} \\ 0 & \text{otherwise} \end{cases}
\end{align*}
Suppose the decomposition into monic irreducibles of $(g,f)$ is $P_1^{\alpha_1}\ldots P_r^{\alpha_r}$ then 

\begin{align*}
a_f(g) &= \alpha_1 \log_q |P_1|+\alpha_2 \log_q P_2 + \ldots + \alpha_r \log_q P_r\\
&=\log_q |P_1^{\alpha_1}\ldots P_r^{\alpha_r}|=\log_q|(f,g)|.
\end{align*}
So 

\begin{equation}
\sum_{\mymod{g}{f}}\log_q|(f,g)|\varepsilon(hg,f)=\sum_{g\mid(h,f)}|g|\Lambda \left(\frac{f}{g}\right).
\end{equation}

We now let $H$ be the Liouville function, $\lambda$, defined by 

$$\lambda(f)=(-1)^{\Omega(f)}$$ where $\Omega(f)$ is the number of monic irreducible divisors of $f$ counted with multiplicity. Note that $\lambda$ is multiplicative: letting $f,g\in A$ be such that $(f,g)=1$, then $ \Omega(fg) = \Omega(f) + \Omega(g) $ and so 

$$\lambda(fg)=(-1)^{\Omega(fg)}=(-1)^{\Omega(f)+\Omega(g)}=(-1)^{\Omega(f)}(-1)^{\Omega(g)}=\lambda(f)\lambda(g).$$ 
Consider $\mathcal{L}(h)=\sum_{d\mid h} \lambda(d) $. Since $\lambda$ is multiplicative, so too is $\mathcal{L}$. We first consider the case where $h$ is a power of some monic irreducible. Then 

\begin{align*}
\mathcal{L}(h) = \sum_{d \mid P^\alpha} \lambda(d)&=\lambda(1)+\lambda(P)+\lambda(P^2)+\ldots+\lambda(P^\alpha)\\
&=1-1+1-\ldots+(-1)^\alpha\\
&= \begin{cases}0 & \text{if $\alpha$ is odd} \\ 1 & \text{if $\alpha$ is even}. \end{cases}
\end{align*}
So now if $h=P_1^{\alpha_1}\ldots P_r^{\alpha_r}$ then 

$$\mathcal{L}(h)=\mathcal{L}(P_1^{\alpha_1})\ldots\mathcal{L}(P_r^{\alpha_r})$$ so if any of the $\alpha_i$ are odd then $\mathcal{L}(h)=0$. If all $\alpha_i$ are even then $\mathcal{L}(h)=1$. Note that all $\alpha_i$ will be even if and only if $h$ is a perfect square. So we have shown that 

\begin{align*}a_f(g)=\begin{cases}
1 &\text{ if $(f,g)$ is a perfect square}\\
0 &\text{ otherwise}
\end{cases}\end{align*} So 

\begin{equation}
\sum_{\substack{\mymod{g}{f}\\(f,g)\text{ square}}} \varepsilon(hg,f) =\sum\limits_{g\mid(h,f)}|g|\lambda\left(\frac{f}{g}\right).
\end{equation}

\section{Summation of some Dirichlet series}
In this section we look at the Dirichlet series for functions of the form (\ref{defn:S}) starting with the following theorem which is the function field analogue of \cite[Theorem 5]{anderson_apostol_1}.

\begin{theorem}
Consider the polynomial zeta function $\zeta_A(s)$ given by $$\zeta_A(s)=\sum\limits_f |f|^{-s}$$ and the Dirichlet series of $H$ given by $$D_H(s)=\sum\limits_f H(f)|f|^{-s}.$$ Suppose that $D_H(s)$ is convergent for $\Re(s)>\sigma_0$.Then, for $\Re(s)>1$ we have
$$\sum_h S(h;f)|h|^{-s} = \zeta_A(s) \sum_{g \mid f}G(g)H(f/g)|g|^{-s}$$ and for $\Re(s)>\sigma_0$ $$ \sum_f S(h;f) |f|^{-s} = D_H(s) \sum_{g \mid h} G(g) \lvert g\rvert^{-s}.$$
\end{theorem}

\begin{proof}
Recall that $S(h;f)=\sum_{g \mid (h,f)} G(g)H(f/g)$. Multiplying by $|h|^{-s}$, writing $h=h'g$ and summing over $h$ gives

\begin{align*}
\sum_h S(h;f)|h|^{-s} &= \sum_{h'}\sum_{g\mid f}G(g)H(f/g)|h'g|^{-s} \\
&= \sum_{h'}|h'|^{-s} \sum_{g \mid f} G(g)H(f/g) |g|^{-s} \\
&= \zeta_A(s) \sum_{g \mid f} G(g) H(f/g)|g|^{-s}.
\end{align*}
Again starting from $S(h;f)=\sum_{g \mid (h,f)} G(g)H(f/g)$ but now multiplying by $|f|^{-s}$, writing $f=f'g$ and summing over $f$ gives

\begin{align*}
\sum_f S(h;f)|f|^{-s} &= \sum_{f'} \sum_{g \mid h} G(g)H(f')|f'g|^{-s} \\
&= \sum_{f'}|f'|^{-s}H(f') \sum_{g \mid h} G(g)|g|^{-s} \\
&= D_H(s) \sum _{g \mid h} G(g)|g|^{-s}. 
\end{align*}
So now we need only prove the convergence criteria. Note that 
$$\sum_{g \mid f} G(g) H(f/g)|g|^{-s} \qquad \text{and} \qquad \sum _{g \mid h} G(g)|g|^{-s}$$ are both finite sums, so we can disregard them. The zeta function is convergent for $\Re(s)>1$ and $D_H(s)$ is convergent for $\Re(s)>\sigma_0$ by hypothesis.
\end{proof}

\vspace{1.0cm}
	
\noindent \textbf{Acknowledgment:} The first author is grateful to the Leverhulme Trust (RPG-2017-320) for the support through the research project grant ``Moments of $L$-functions in Function Fields and Random Matrix Theory". The second author was partially supported by an Undergraduate Research Bursary from the London Mathematical Society and the University of Exeter. 


\end{document}